\documentclass[12pt]{amsart}
\usepackage{url, calc, bm,bbm}
\usepackage{amsmath,amsxtra,amssymb,latexsym,epsfig,amscd,amsthm,fancybox,epsfig}
\usepackage[mathscr]{eucal}
\usepackage{graphicx}
\usepackage{multicol,xcolor}
\usepackage{epsfig} % for postscript graphics files
\usepackage{epstopdf}
\usepackage{cases}
\usepackage{subfig}
\usepackage{color}
\usepackage{hyperref}
\usepackage{bigints}
\usepackage{enumerate}
\usepackage{pdfpages}
\DeclareMathOperator*{\argmax}{arg\,max}
\DeclareMathOperator*{\argmin}{arg\,min}

\graphicspath{{./Figures/}}

\usepackage{bm}

\usepackage[normalem]{ulem}

\usepackage{tikz}
\usetikzlibrary{intersections}

\usepackage{thmtools,thm-restate}

\newtheorem{thm}{Theorem}[section]
\newtheorem {asp}{Assumption}[section]

\newtheorem{lm}{Lemma}[section]
\newtheorem{rmk}{Remark}[section]

\newtheorem{prop}{Proposition}[section]
\theoremstyle{definition}

\theoremstyle{remark}

\numberwithin{equation}{section}

\newcommand{\eps}{\varepsilon}

\newcommand{\E}{\mathbb{E}}

\newcommand{\PP}{\mathbb{P}}

\newcommand{\R}{\mathbb{R}}

\numberwithin{equation}{section}

\newcommand{\bed}{\begin{displaymath}}
\newcommand{\eed}{\end{displaymath}}
\newcommand{\bea}{\bed\begin{array}{rl}}
\newcommand{\eea}{\end{array}\eed}

\newcommand{\barray}{\begin{array}{ll}}
\newcommand{\earray}{\end{array}}

\def\bar{\overline}
\def\hat{\widehat}
\def\a.s{\text{\;a.s.\;}}

\begin{document}
\title[Optimal sustainable harvesting]{Optimal sustainable harvesting of populations in random environments}

\author[L. H. R. Alvarez E.]{Luis H. R. Alvarez E.}
\address{Department of Accounting and Finance\\
 Turku School of Economics\\
 FIN-20014 University of Turku\\
Finland}
 \email{lhralv@utu.fi}

\author[A. Hening]{Alexandru Hening}
\address{Department of Mathematics\\
Tufts University\\
Bromfield-Pearson Hall\\
503 Boston Avenue\\
Medford, MA 02155\\
United States
}
\email{alexandru.hening@tufts.edu}


\maketitle

\begin{abstract}
We study the optimal sustainable harvesting of a population that lives in a random environment. The novelty of our setting is that we maximize the asymptotic harvesting yield, both in an expected value and almost sure sense, for a large class of harvesting strategies and unstructured population models. We prove under relatively weak assumptions that there exists a unique optimal harvesting strategy characterized by an optimal threshold below which the population is maintained at all times by utilizing a local time push-type policy. We also discuss, through Abelian limits, how our results are related to the optimal harvesting strategies when one maximizes the expected cumulative present value of the harvesting yield and establish a simple connection and ordering between the values and optimal boundaries. Finally, we explicitly characterize the optimal harvesting strategies in two different cases, one of which is the celebrated stochastic Verhulst Pearl logistic model of population growth.
\end{abstract}

\tableofcontents

\section{Introduction}\label{s:intro}

When trying to establish the best harvesting policy of a certain species, one needs to take into account both the biological and economic implications. It is well known that overharvesting might lead to the extinction of whole populations (see \cite{Cla10, G71, P06, LES95}). Many species of animals (birds, mammals, and fish) are endangered because of unrestricted harvesting or hunting. In some instances people have overestimated the population density of a certain species, and since it takes a while for a harvested population to recover to previous levels, this has led to either local or global extinctions. However, if we underharvest a species, this can lead to the loss of valuable resources. We are therefore presented with a conundrum: should we overharvest and gain economically but possibly drive a species extinct or should we underharvest to make sure extinction is less likely but lose precious resources? We present a model and a harvesting method which give us, based on a rigorous mathematical analysis, the best possible {\em sustainable harvesting policy that does not drive the species extinct}.

We study a population whose dynamics is continuous in time and that is affected by both biotic (competition) and abiotic (rainfall, temperature, resource availability) factors. Since the abiotic factors are affected by random disturbances, we look at a model that has environmental stochasticity. This transforms a system that is modeled by an ordinary differential equation (ODE) into a system that is modeled by a stochastic differential equation (SDE). We refer the reader to \cite{T77} for a thorough discussion of environmental stochasticity.

We build on the results from \cite{AS98, Al01} and \cite{HNUW18}. Suppose that in an infinitesimal time $dt$ we harvest a quantity $dZ_t$ where $(Z_t)_{t\geq 0}$ is any adapted, non-negative, non-decreasing, and right continuous process. We determine the \textit{optimal harvesting strategy} maximizing the \textit{expected average asymptotic yield}
$$
\ell=\liminf_{T\rightarrow\infty} \mathbb{E}_x\frac{1}{T}\int_0^T\,dZ_t=\liminf_{T\rightarrow\infty} \frac{\mathbb{E}_x Z_T}{T}
$$
of harvested individuals. As in \cite{HNUW18}, and in contrast to what happens in a significant part of the literature (see \cite{LES94,LES95, AS98, LO97}), the optimal strategy will be such that the population is never depleted and cannot be harvested to extinction. This is clear since if $Z_T\rightarrow 0$ in some sense then $\ell=0$ in the above equation.  Our main result is that the optimal harvesting strategy is of the local time reflection type: the population is kept in the interval $(0,b^\ast]$ at all times by first harvesting $(x-b^\ast)^+$ and then harvesting only when the population hits the boundary just enough to maintain the population density below $b^\ast$. This result was conjectured in \cite{HNUW18} where the authors showed that if the harvesting rate is bounded the optimal strategy is of bang-bang type i.e. there is a threshold $x^*>0$ such that if the population size is under $x^*$ there is no harvesting while if the population size is above $x^*$ we harvest according to the maximal rate $M>0$. If one works with discounted yields like in \cite{AS98}, then interestingly the optimal harvesting strategy is also of this local time reflection type. In our setting the diffusion governing the unharvested population is much more general than the one from \cite{AS98} and \cite{HNUW18} where the authors mostly work with a stochastic Verhulst-Pearl diffusion or its generalization. In the current paper we present a unifying result that encompasses a large variety of stochastic models.

Another advantage of our framework is that it does not depend on parameters that are hard to be quantified empirically. Many papers from the literature (see \cite{AS98}) work with a time discounted yield in order to capture the opportunity cost of capital. However, it is a difficult question to come up with a good value for the discount factors (see \cite{dreze1987theory}). Moreover, as \cite{LES94} state in their influential paper focusing on the relationship between discounting and extinction risk:
\begin{quote}``Thus, even when the discount rate is less than the critical value predicted by deterministic models, the economically optimal strategy will often be immediate harvesting to extinction. These results make a powerful a argument that, for the common good, economic discounting should be avoided in the development of optimal strategies for sustainable use of biological resources."
\end{quote}

Our model does not involve any discount factors. We generalize the setting of \cite{HNUW18} where the authors assumed the harvesting rate was bounded by some parameter $M>0$. This corresponds to having total control over the harvested population. Moreover, it also side-steps needing to know the parameter $M>0$ which could be hard to estimate realistically.

The paper is organized as follows. In Section \ref{s:model} we introduce the model and prove the main results. Section \ref{s:discount} showcases how our model relates to the discounted model from \cite{AS98,Al01}. In particular we show that by letting the discount rate go to zero, $r\downarrow 0$, we can recover in a sense the results of this paper. In Section \ref{s:applications} we look at two explicit applications of our results. As a first application we look at the Verhulst-Pearl diffusion model studied in \cite{AS98,HNUW18}. The second model we analyze is the one studied in \cite{A00,LO97}. Finally, Section \ref{s:discussion} is dedicated to a discussion of our results.

\section{Model and Results}\label{s:model}
We consider a population whose density $X_{t}$ at time $t\geq 0$ follows, in the absence of harvesting, the stochastic differential equation (SDE)
\begin{equation}\label{e:1d}
d X_{t} =  X_{t}\mu(X_{t})\,dt + \sigma  (X_{t})\,dB_{t},
\end{equation}
where $(B_{t})_{t\geq 0}$ is a standard one dimensional Brownian motion. This describes a population $X$ with per-capita growth rate given by $\mu(x)>0$ and infinitesimal variance of fluctuations in the per-capita growth rate given by $\sigma^2(x)/x^2$ when the density is $X_{t}=x$.
We make the following standing assumption throughout the paper.
\begin{asp}\label{A:ES}
The functions $\mu,\sigma:(0,\infty)\rightarrow \R$ are continuous and satisfy the \textit{Engelbert-Schmidt} conditions:
$$
\sigma(x)>0~\text{and}~\exists \varepsilon>0 ~\text{s.t.}~ \int_{x-\varepsilon}^{x+\varepsilon} \frac{1+|y\mu(y)|}{\sigma^2(y)}\,dy<\infty ~\text{for any}~x\in(0,\infty).
$$
\end{asp}
These conditions ensure the existence and uniqueness of weak solutions to \eqref{e:1d} (see for example \cite{ES91}). In addition, we want the population to persist in the absence of harvesting and to not explode to infinity (which would be absurd from a biological point of view). To this end we will assume throughout our analysis that the boundaries of the state space of the population density are unattainable (i.e. either natural or entrance) for $X$ in the absence of harvesting. This means that even though the process may tend towards a boundary it will never attain it in finite time. We refer the reader to Section II.6 from \cite{BS15} for a thorough discussion of the boundary classification of one-dimensional diffusions.

We denote the density of the scale function of $X$ by
\begin{equation}\label{e:scale}
S'(x)=\exp\left(-\int_c^x\frac{2\mu(y)y}{\sigma^2(y)}\,dy\right),
\end{equation}
where $c\in \mathbb{R}_+$ is an arbitrary constant. The density of the speed measure $m$ is, in turn, denoted by
\begin{equation}\label{e:speed}
m'(x)=\frac{2}{\sigma^2(x)S'(x)}.
\end{equation}
The second order differential operator
\begin{equation}\label{e:generator}
\mathcal{A}:=\frac{1}{2}\sigma^2(x)\frac{d^2}{dx^2}+\mu(x)x\frac{d}{dx}=\frac{1}{2}\frac{d}{dm} \frac{d}{dS}
\end{equation}
is the infinitesimal generator of the underlying diffusion $X$.

For most harvesting applications it is sufficient to make the following assumption.

%% \begin{asp}\label{A:ECO} The functions $\sigma, \mu:[0,\infty)\rightarrow \R$ satisfy:
%% \begin{itemize}
%% \item $\sigma(x)=x\tilde\sigma, x\in [0,\infty)$ for some $\tilde \sigma>0$.
%% \item $\mu(0)-\frac{\tilde \sigma^2}{2}>0$.
%% \item $\tilde \sigma$ is locally Lipschitz.
%% \item As $x\rightarrow \infty$ we have $\mu(x)\rightarrow-\infty$.
%% \end{itemize}
%% \end{asp}

\begin{asp}\label{A:suff} ~
\newline
\begin{itemize}
\item[(A1)] The function $\mu$ is nonincreasing and fulfills the limiting conditions $\lim_{x\rightarrow 0+}\mu(x) > \eta$ and $\lim_{x\rightarrow\infty}\mu(x)<-\eta$ for some $\eta >0$.
\item[(A2)] The function $\mu(x)x$ has a unique maximum point $\hat{x}=\argmax\{\mu(x)x\}$ so that $\mu(x)x$ is increasing on $(0,\hat{x})$
and decreasing on $(\hat{x},\infty)$.
\item[(A3)] $\lim_{x\rightarrow 0+}m((x,y))<\infty$ for $x< y$.
\end{itemize}
\end{asp}
\begin{rmk}
It is worth pointing out that assumption (A1) guarantees that the per-capita growth rate $\mu$ vanishes at some given point $x_0=\mu^{-1}(0)$. In typical population models this point coincides with the carrying capacity of the population. We naturally have that $\hat{x}<x_0$.

Condition (A3) is needed for the existence of a stationary distribution for the process $X$. Under our boundary assumptions, it guarantees that $0$ is repelling for $X$ and the condition $\lim_{x\rightarrow 0+}S((x,y))=+\infty$, for $x< y$, is satisfied (cf. p. 234 in \cite{KaTa81}).
\end{rmk}

A stochastic process $(Z_t)_{t\geq 0}$ taking values in $[0,\infty)$
is said to be an \textit{admissible harvesting strategy} if
$(Z_t)_{t\geq 0}$ is non-negative, nondecreasing, right continuous, and adapted to the filtration $(\mathcal{F}_t)_{t\geq 0}$ generated by the driving  Brownian motion $(B_{t})_{t\geq 0}$. We denote the class of all admissible harvesting strategies (or controls) by $\Lambda$. Assume that $(Z_t)_{t\geq 0}\in \Lambda$ and that at time $t$ we harvest in the infinitesimal period $dt$ an amount $dZ_t$. Then our harvested population's dynamics is given by
\begin{equation}\label{e:harvested}
d X^{Z}_{t} =  X^{Z}_{t}\mu(X^{Z}_{t})\,dt + \sigma(X^{Z}_{t})\,dB_{t} -dZ_t, ~X^Z_0=x>0.
\end{equation}
We consider the following ergodic singular control problem:
\begin{equation}\label{e:max_control}
\sup_{Z\in \Lambda} \liminf_{T\rightarrow\infty} \frac{1}{T}\mathbb{E}_x\int_0^TdZ_s.
\end{equation}
We are interested (as in \cite{HNUW18}) in the maximization of the \textit{expected asymptotic harvesting yield} (also called the \textit{expected average cumulative yield})  of the population.

Before presenting our main findings on the optimal ergodic harvesting strategy and the maximal expected average cumulative yield we first establish the following
auxiliary verification lemma.
\begin{lm}\label{aux1}
Let $\ell$ be a given positive constant and assume that $v:\mathbb{R}_+\mapsto \mathbb{R}_+$ is a twice continuously differentiable function
satisfying the inequalities $v'(x)\geq 1$ and $(\mathcal{A}v)(x)\leq \ell$ for all $x\in \mathbb{R}_+$. Then
$$
\liminf_{T\rightarrow\infty} \frac{1}{T}\mathbb{E}_x\int_0^TdZ_s\leq \ell
$$
for all $Z\in \Lambda$.
\end{lm}
\begin{proof}
Applying the generalized It{\^o}-D{\"o}blin (see \cite{Harrison85}) change of variable formula to the nonnegative function $v$ yields
\begin{align*}
0\leq v(X_{T_n}^Z) = v(x) &+ \int_0^{T_n}(\mathcal{A}v)(X_s^Z)ds + \int_0^{T_n}\sigma(X_s^Z)v'(X_s^Z)dB_s\\
 &-\int_0^{T_n}v'(X_s^Z)dZ_s^c + \sum_{s\leq T_n}(v(X_s^Z)-v(X_{s-}^Z)),
\end{align*}
where $Z^c$ denotes the continuous part of an arbitrary admissible harvesting strategy $Z\in \Lambda$, $T>0$, and $T_n = T\wedge \inf\{t\geq 0:X_t^Z\geq n\}$ is a sequence of finite stopping times converging to $T$ as $n\rightarrow\infty$. Reordering terms and taking expectations shows that
\begin{align*}
\mathbb{E}_x\int_0^{T_n}v'(X_s^Z)dZ_s^c + \mathbb{E}_x\sum_{s\leq T_n}\int_{X_{s}^Z}^{X_{s-}^Z}v'(y)dy \leq v(x) + \mathbb{E}_x\int_0^{T_n}(\mathcal{A}v)(X_s^Z)ds.
\end{align*}
Imposing now the inequalities $v'(x)\geq 1$ and $(\mathcal{A}v)(x)\leq \ell$ demonstrates that
\begin{align*}
\mathbb{E}_x\int_0^{T_n}dZ_s \leq v(x) + \ell T_n.
\end{align*}
Letting $n\rightarrow \infty$ and applying Fatou's lemma then finally yields
\begin{align*}
\mathbb{E}_x\int_0^{T}dZ_s \leq v(x) + \ell T
\end{align*}
from which the alleged results follow.
\end{proof}

It is natural to ask if there is a function $v$ and a constant $\ell^\ast$ satisfying the conditions of Lemma \ref{aux1}. In order to show that the answer to
this question is positive, we now follow the seminal paper \cite{Kar83} and investigate the following question:
can we find two constants $\ell^\ast , b^\ast$ and a twice continuously differentiable function $u(x)$ satisfying the conditions
\begin{align}
\begin{split}\label{freebdry}
\lim_{x\downarrow 0}\frac{u'(x)}{S'(x)} &=0\\
(\mathcal{A}u)(x) &=\ell^\ast,\quad x\in(0,b^\ast),\\
u'(x)&=1,\quad x\geq b^\ast.
\end{split}
\end{align}
Using \eqref{e:generator} we get
$$
\frac{d}{dx}\left(\frac{u'(x)}{S'(x)}\right)=(\mathcal{A}u)(x)m'(x)=\ell^\ast m'(x), ~x\in(0,b^*).
$$
The last equality, together with the boundary condition $\lim_{x\downarrow 0}u'(x)/S'(x)=0$, yields that
\begin{equation}\label{e:u}
\begin{aligned}
u'(x)
&= \begin{cases}
\ell^\ast S'(x)m((0,x)), & \mbox{if $x\in(0,b^*)$}\\
1 & \mbox{if $x\geq b^*$}
\end{cases}
\end{aligned}
\end{equation}

Invoking the twice continuous differentiability of $u$ across the boundary $b^\ast$, and noting that $u''(b^*)=0$ then shows that the constants $\ell^\ast , b^\ast$ are the solutions of the system
\begin{align}\label{opt1}
\ell^\ast = \mu(b^\ast)b^\ast = \frac{1}{S'(b^\ast)m((0,b^\ast))}.
\end{align}
We can now establish the following.
\begin{lm}\label{aux2}
The optimality conditions \eqref{opt1} have a unique solution and for all $Z\in \Lambda$
$$
\liminf_{T\rightarrow\infty} \frac{1}{T}\mathbb{E}_x\int_0^TdZ_s\leq \ell^\ast = \sup_{b>0}\left\{\frac{1}{S'(b)m((0,b))}\right\}=\frac{1}{S'(b^*)m((0,b^*))}
$$
where $b^*$ is the unique zero of
$$
f(x) = \int_0^x(\mu(y)y-\mu(x)x)m'(y)dy.
$$
Futhermore, the function $u$ defined by \eqref{freebdry} satisfies the conditions of Lemma \ref{aux1}.
\end{lm}
\begin{proof}
We first show that the optimality conditions \eqref{opt1} have a unique solution under our assumptions.
To this end we investigate the behavior of the continuous function $f:(0,\infty)\rightarrow \R$ defined by
$$
f(x) := \frac{1}{S'(x)}-\mu(x)x m((0,x)).
$$
Making use of Assumption (A3) guarantees that
$$
\int_0^x\mu(y)y m'(y) dy =\int_0^x \left(\frac{1}{S'(z)}\right)'\,dz= \frac{1}{S'(x)} - \lim_{y\rightarrow 0+}\frac{1}{S'(y)}= \frac{1}{S'(x)}.
$$
Therefore, we can express $f(x)$ as
$$
f(x) = \int_0^x(\mu(y)y-\mu(x)x)m'(y)dy.
$$
It is clear that $f(\hat{x})<0$ and
$$
f(x_0)=\int_0^{x_0}\mu(y)ym'(y)dy>0
$$
demonstrating, using the intermediate value theorem, that $f$ has at least one root $b^\ast\in(\hat{x},x_0)$. To prove that the root is unique, we notice that
if $y>x$, then
\begin{align*}
f(y)-f(x) &=\frac{1}{S'(y)}-\mu(y)y m((0,y)) - \left(\frac{1}{S'(x)}-\mu(x)x m((0,x))\right)\\
&=\int_x^y \mu(t)tm'(t)dt-\mu(y)y m((0,y))+\mu(x)x m((0,x))\\
&=\int_x^y (\mu(t)t-\mu(y)y)m'(t)dt + (\mu(x)x-\mu(y)y) m((0,x)).
\end{align*}
Hence, if $x<y\leq \hat{x}$ then $f(y)-f(x) < 0$ proving that $f$ is strictly decreasing on $(0,\hat{x})$. If, in turn,
$\hat{x}\leq x < y$ then $f(y)-f(x) > 0$ proving that $f$ is strictly increasing on $(\hat{x},\infty)$. Combining these observations with
the continuity of $f$ and the fact that $\lim_{x\rightarrow 0+}f(x)=0$ then proves that
the root $b^\ast\in(\hat{x},x_0)$ is unique and, consequently, that a unique pair $\ell^\ast,b^\ast$ exists. Moreover, since
$$
\frac{d}{db}\left[\frac{1}{S'(b)m((0,b))}\right]=\frac{-2f(b)}{\sigma^2(b)S'(b)m^2((0,b))}
$$
we get that
$$
b^\ast = \argmax\left\{\frac{1}{S'(b)m((0,b))}\right\}
$$
and
$$
\ell^\ast=\frac{1}{S'(b^\ast)m((0,b^\ast))}=\mu(b^\ast)b^\ast.
$$

We now prove that the function $u$ satisfies the conditions of Lemma \ref{aux1}.
We first observe that
$(\mathcal{A}u)(x) = \mu(x)x$ for all $x\in[b^\ast,\infty)$. Since $\mu(x)x$ is decreasing on $(\hat{x},\infty)$ and $b^\ast > \hat{x}$ we find that
$(\mathcal{A}u)(x) \leq \ell^\ast = \mu(b^\ast)b^\ast$ for all $x\in\mathbb{R}_+$. On the other hand, since
$$
u''(x) = \frac{2S'(x)\ell^\ast}{\sigma^2(x)} \left(\frac{1}{S'(x)}-\mu(x)xm((0,x))\right) = \frac{2S'(x)\ell^\ast}{\sigma^2(x)}f(x) < 0
$$
for all $x<b^\ast$ and $u'(b^\ast)=1$ we find that $u'(x)\geq 1$ for all $x\in \mathbb{R}_+$. The last alleged claim now follows from Lemma \ref{aux1}.
\end{proof}
\begin{rmk}
Lemma \ref{aux2} also shows that the function $u(x)$ satisfying the considered free boundary value problem is concave on $\mathbb{R}_+$.
This property is later shown to be the principal determinant of the sign of the impact of increased volatility on the optimal harvesting policy and the expected average cumulative yield.
\end{rmk}

Lemma \ref{aux2} essentially shows that if there is an admissible harvesting strategy resulting into a value satisfying the variational inequalities of Lemma \ref{aux1}, then the value of that
policy dominates the value of the maximal expected average cumulative yield. Naturally, if we could determine an admissible policy yielding precisely the value characterized in Lemma \ref{aux2}
then that policy would automatically constitute the optimal harvesting policy. This is accomplished in the following theorem summarizing our main result on the optimal sustainable harvesting policy.
\begin{restatable}{thm}{main}\label{t:main}
Suppose Assumptions \ref{A:ES} and \ref{A:suff} hold and $X^Z_0=x>0$. The optimal harvesting strategy is
\begin{equation}\label{e:optimal_strategy}
\begin{aligned}
Z_t^{b^\ast}
&= \begin{cases}
\left(x-b^\ast\right)^+ & \mbox{if $t=0$,} \\
L(t,b^\ast) & \mbox{if $t>0$}
\end{cases}
\end{aligned}
\end{equation}
where $L(t,b^\ast)$ is the {\em local time push} of the process $X^Z$ at the boundary $b^\ast$ (cf. \cite{Harrison85,Kar83,ShLeGa84}). The optimal harvesting boundary $b^\ast$ as well as
the maximal expected average asymptotic  yield $\ell^\ast$ are the solutions of the optimality conditions
$$
\ell^\ast = \mu(b^\ast)b^\ast = \frac{1}{S'(b^\ast)m((0,b^\ast))}.
$$
Moreover,
$$
\sup_{Z\in \Lambda} \liminf_{T\rightarrow\infty} \frac{1}{T}\mathbb{E}_x\int_0^TdZ_s = \lim_{T\rightarrow\infty} \frac{\mathbb{E}_x[Z_T^{b^\ast}]}{T} = \ell^\ast = \mu(b^\ast)b^\ast.
$$
\end{restatable}
\begin{proof}
It is clear that the proposed harvesting strategy $Z^{b^\ast}$ is admissible. Our objective is now to show that this policy attains the maximal expected average cumulative yield $\ell^\ast$ and is, therefore, optimal. To show that this is indeed the case, we first notice that the harvesting policy $Z^{b^\ast}$ is continuous on $t>0$, increases only when $X_t^{Z^{b^\ast}}=b^\ast$, and maintains the process $X_t^{Z^{b^\ast}}$ in $(0,b^\ast]$ for all $t>0$ (\cite{Harrison85,Kar83,ShLeGa84}).  In this case \eqref{e:harvested} can be re-expressed as
\begin{equation*}
Z_T^{b^\ast} = x - X_T^{Z^{b^\ast}} + \int_0^T \mu\left(X_t^{Z^{b^\ast}}\right)X_t^{Z^{b^\ast}}\,dt+ \int_0^T \sigma \left(X_t^{Z^{b^\ast}}\right)\,dB_{t}.
\end{equation*}
The continuity of the diffusion coefficient $\sigma(x)$ now guarantees that $\sigma(X_t^{Z^{b^\ast}})$ is bounded for all $t>0$ and, therefore, that
\begin{align*}
\frac{\mathbb{E}_x\left[Z_T^{b^\ast}\right]}{T} = \frac{x - \mathbb{E}_x\left[X_T^{Z^{b^\ast}}\right]}{T} + \frac{1}{T}\mathbb{E}_x\int_0^T \mu\left(X_t^{Z^{b^\ast}}\right)X_t^{Z^{b^\ast}}\,dt.
\end{align*}
Consequently,
\begin{align*}
\lim_{T\rightarrow\infty}\frac{\mathbb{E}_x\left[Z_T^{b^\ast}\right]}{T} = \lim_{T\rightarrow\infty} \frac{1}{T}\mathbb{E}_x\int_0^T \mu\left(X_t^{Z^{b^\ast}}\right)X_t^{Z^{b^\ast}}\,dt.
\end{align*}
Since $m((0,b^\ast))<\infty$ we notice that the process is ergodic and has an invariant probability measure $\pi(\cdot)=\frac{m(\cdot)}{m((0,b^\ast))}$ (cf. \cite{BS15}, pp. 37-38). Hence,
\begin{align*}
\lim_{T\rightarrow\infty}\frac{1}{T}\mathbb{E}_x\int_0^T \mu\left(X_t^{Z^{b^\ast}}\right)X_t^{Z^{b^\ast}}\,dt = \int_0^{b^\ast} \mu(x)x \frac{m'(x)}{m((0,b^\ast))}\,dx = \mu(b^\ast)b^\ast =  \ell^\ast,
\end{align*}
demonstrating the optimality of the proposed policy.
\end{proof}
\begin{rmk}
It is worth noticing that since (cf. pp. 36--38 in \cite{BS15})
$$
\lim_{t\rightarrow \infty}\frac{\int_0^t\mu(X_{s})X_{s}\mathbbm{1}_{(0,b]}(X_{s})ds}{\int_0^t\mathbbm{1}_{(0,b]}(X_{s})ds}=\frac{1}{S'(b)m((0,b))}=\mu(b)b
$$
our findings are in line with observations based on renewal theoretic approaches to ergodic control (cf. Chapter 5 in \cite{Harrison85}). On the other hand
we also observe that
$$
b^\ast = \argmax_{b\in \mathbb{R}_+}\left\{\E\left[\mu(X^{Z^{b}}_\infty)X^{Z^{b}}_\infty\right]\right\}
$$
where $X_t$ denotes the population density in the absence of harvesting and $X^{Z^{b}}_t\rightarrow X^{Z^{b}}_\infty\sim m'(x)\mathbbm{1}_{(0,b]}(x)/m((0,b))$ as $t\uparrow\infty$. Consequently, the same conclusion could be obtained by focusing on the ergodic limit of the process controlled by $Z_t^{b}$.
\end{rmk}
Theorem \ref{t:main} demonstrates that the optimal harvesting policy is of the standard local time push type in the ergodic control setting as well. Consequently, under the optimal harvesting policy the population is maintained below an optimal threshold by harvesting (in an infinitely intense fashion) only at instants when the population hits the optimal boundary. Below the critical threshold the population is naturally left unharvested.

One may wonder wether the findings of Theorem \ref{t:main} could be extended further to a setting focusing on the almost sure maximization problem
\begin{equation}\label{e:almostsure}
\sup_{Z\in\Lambda}\liminf_{T\rightarrow\infty} \frac{1}{T}\int_0^T\,dZ_t=\sup_{Z\in\Lambda}\liminf_{T\rightarrow\infty} \frac{ Z_T}{T}.
\end{equation}
This is an almost sure statement, compared to the maximization from \eqref{e:max_control} which deals with expected values. In order to delineate general circumstances under which the almost sure maximization problem admits a local time push type solution, we initially analyze the problem by focusing solely on this type of harvesting policies. Our main findings on that class is established in the next proposition.
\begin{prop}\label{Prob1Limit}
Let $Z^b\in \Lambda$ be an arbitrary local time push type harvesting policy maintaining the population density on $(0,b)$ for all $t>0$. Then, for any $X_0^Z=x\in (0,b)$
\begin{equation}\label{AlmCertlim1}
\PP_x\left\{\lim_{T\rightarrow\infty}\frac{Z_T^{b}}{T}=\lim_{T\rightarrow\infty}\frac{1}{T}\int_0^T \mu\left(X_t^{Z^{b}}\right)X_t^{Z^{b}}\,dt = \frac{1}{S'(b)m((0,b))}\right\}=1.
\end{equation}
Consequently,
\begin{equation}\label{OptLim}
\PP_x\left\{\lim_{T\rightarrow\infty}\frac{Z_T^{b}}{T} \leq \lim_{T\rightarrow\infty}\frac{Z_T^{b^\ast}}{T}=\sup_{b>0}\left\{\frac{1}{S'(b)m((0,b))}\right\} = \mu(b^\ast)b^\ast\right\}=1.
\end{equation}
\end{prop}
\begin{proof}
Let $b\in (0,\infty)$ be an arbitrary finite boundary and consider the policy $Z_t^b\in \Lambda$ maintaining the population density in $(0,b)$ for all $t>0$. As in the case of Theorem \ref{t:main}, the policy is continuous on $t>0$ and increases only when $X_t^{Z^{b}}=b$. Moreover,
\begin{equation}\label{e:Z}
\frac{Z_T^{b}}{T} = \frac{x}{T}+ \frac{1}{T}\int_0^T \mu\left(X_t^{Z^{b}}\right)X_t^{Z^{b}}\,dt+ \frac{1}{T}\int_0^T \sigma \left(X_t^{Z^{b}}\right)\,dB_{t} -\frac{ X_T^{Z^{b}}}{T}.
\end{equation}
Since $|X_t^{Z^{b}}|\leq b$ for all $t>0$ one trivially has
\begin{equation}\label{e:bounded}
\lim_{T\rightarrow\infty}\frac{X_T^{Z^{b}}}{T}=0
\end{equation}
with probability $1$. Since $m((0,b))<\infty$ the controlled process is ergodic on $(0,b)$ and has an invariant probability measure $\pi(\cdot)=\frac{m(\cdot)}{m((0,b))}$. Invoking the ergodic results from \cite{BS15} (pp. 37-38) shows that almost surely
\begin{equation}\label{e:ergodic}
\lim_{T\rightarrow\infty}\frac{1}{T}\int_0^T \mu\left(X_t^{Z^{b}}\right)X_t^{Z^{b}}\,dt = \int_0^{b} \mu(x)x \frac{m'(x)}{m((0,b))}\,dx = \frac{1}{S'(b)m((0,b))}.
\end{equation}
Let $L_{T}=\int_0^T \sigma \left(X_t^{Z^{b}}\right)\,dB_{t}$. Then $(L_{T})_{T\geq 0}$ is a local martingale with quadratic variation $Q_{T}=\int_0^T \sigma^2 \left(X_t^{Z^{b}}\right)\,dt$. By the ergodic results from \cite{BS15}, one has that almost surely
$$
\limsup_{T\rightarrow\infty}\frac{Q_{t}}{T}= \int_0^{b} \sigma^2(x)\frac{m'(x)}{m(0,b)}\,dx<\infty.
$$
This combined with the Law of Large Numbers for local martingales (see Theorem 1.3.4 from \cite{MAO}) yields that almost surely
\begin{equation}\label{e:LLN}
\lim_{T\rightarrow\infty}\frac{1}{T}\int_0^T \sigma \left(X_t^{Z^{b}}\right)\,dB_{t}=0.
\end{equation}
Using \eqref{e:bounded}, \eqref{e:ergodic} and \eqref{e:LLN} in \eqref{e:Z} we get that \eqref{AlmCertlim1} holds almost surely.
\eqref{OptLim} then follows from Lemma \ref{aux2}.
\end{proof}
Proposition \ref{Prob1Limit} shows that the local time push controls affect the dynamics of the controlled population density in a way
where the almost sure asymptotic average cumulative harvest can be computed explicitly in terms of the exogenous harvesting boundary. Since this
representation is valid for all local time push controls, we find that choosing the threshold according to the rule maximizing the long run expected
average cumulative harvest results in a maximal representation in this setting as well. Given the generality of admissible harvesting strategies,
it is a challenging task to prove a general verification
lemma analogous to Lemma \ref{aux1}. Fortunately, there is a relatively large class of processes for which the desired result is  valid.
To see that this is indeed the case, we first establish the following auxiliary result.
\begin{lm}\label{l:comparison}
Assume \eqref{e:1d} has pathwise unique solutions and there exists an increasing function $\rho:\R_+\rightarrow\R$ such that $|\sigma(x)-\sigma(y)|\leq \rho(|x-y|)$ for all $x,y\in (0,\infty)$ and $\int_{0+}\rho^{-2}(z)\,dz=+\infty$. Suppose $X$ is the solution to \eqref{e:1d} and $X^Z$ is the solution to \eqref{e:harvested} for a fixed $Z\in\Lambda$. If $X_0^Z\leq X_0$ then almost surely
$$
\PP\{X_s^Z\leq X_s,~s\geq 0\}=1.
$$
\end{lm}
\begin{proof}
This is a modification of the arguments from the seminal papers \cite{Ya73} and \cite{IW77} for the comparison of one-dimensional diffusions and the paper \cite{Y86} for the comparison of semimartingales. For small $\eps>0$ define the process $X^\eps$ via
$$
d X^{\eps}_{t} =  X^{\eps}_{t}(\mu(X^{\eps}_{t})+\eps)\,dt + \sigma  (X^{\eps}_{t})\,dB_{t}.
$$
Assume that $X_0=X_0^\eps$. By \cite{IW77} we see that almost surely
$$
X_s\leq X^\eps_s,~s\geq 0.
$$
By the pathwise uniqueness of solutions of \eqref{e:1d}, combined with the continuity of $\mu$ we have almost surely that
$$
X_s= \lim_{\eps\downarrow 0}X^\eps_s,~s\geq 0.
$$
We note that the semimartingales $X^Z$ and $X^\eps$ satisfy the assumptions of Theorem 1 from \cite{Y86}. Therefore, if $X_0^Z\leq X^\eps_0$, we have that almost surely
$$
X_s^Z\leq X^\eps_s, ~s\geq 0.
$$
Taking the limit as $\eps\downarrow 0$ we get
$$
X_s^Z\leq X_s, ~s\geq 0
$$
which finishes the proof.
\end{proof}
\begin{rmk}
We make two remarks on the assumptions needed in Lemma \ref{l:comparison}. First of all, sufficient conditions for pathwise uniqueness of solutions can be found, for example, in \cite{IW,K02,MAO}. Second, for most models of natural resources $\sigma(x)=\sigma x$ for some $\sigma>0$. In those cases the required growth condition is satisfied by simply taking $\rho(x)=\sigma(x)=\sigma x$.
\end{rmk}
Lemma \ref{l:comparison} states a set of conditions under which the solution of the uncontrolled dynamics \eqref{e:1d} dominates the solution of the dynamics subject to harvesting \eqref{e:harvested}.
It is worth pointing out that similar comparison results have been previously established for Lipschitz-continuous coefficients (see Theorem 54 in \cite{Pr05}). However, that result does not directly apply to our setting, since most applied population models have only locally Lipschitz-continuous coefficients. Given our findings in Lemma \ref{l:comparison} we can now establish the following Theorem which extends our results on the expected average cumulative yield to the almost sure setting.
\begin{restatable}{thm}{maina}\label{t:main_as}
Assume that Assumptions \ref{A:ES} and \ref{A:suff} hold, that \eqref{e:1d} has pathwise unique solutions, that there exists an increasing function $\rho:\R_+\rightarrow\R$ such that $|\sigma(x)-\sigma(y)|\leq \rho(|x-y|)$ and $\int_{0+}\rho^{-2}(z)\,dz=+\infty$, and that
\begin{enumerate}
  \item The process $X$ from \eqref{e:1d} has a unique invariant probability measure on $(0,\infty)$.
  \item One can find a twice continuously differentiable function $v:\mathbb{R}_+\mapsto \mathbb{R}_+$ satisfying the variational inequalities $v'(x)\geq 1$ and $(\mathcal{A}v)(x)\leq \ell$ for all $x\in \mathbb{R}_+$.
  \item The function $g(x):=\sigma(x) v'(x)$ is non-decreasing and square-integrable with respect to the speed measure of $X$.
\end{enumerate}
Then for any admissible strategy $Z\in\Lambda$ and any $X^Z_0=x\in(0,\infty)$
\begin{align}\label{dom}
\PP_x\left\{\liminf_{T\rightarrow\infty}\frac{Z_T}{T}\leq \ell\right\}=1.
\end{align}
Moreover,
\begin{align}\label{domopt}
\PP_x\left\{\liminf_{T\rightarrow\infty}\frac{Z_T}{T}\leq \liminf_{T\rightarrow\infty}\frac{Z_T^{b^\ast}}{T} = \ell^\ast=\mu(b^\ast)b^\ast\right\}=1
\end{align}
for all $Z\in \Lambda$ and all $X^Z_0=x\in(0,\infty)$.
\end{restatable}
\begin{proof}
It is clear that for any admissible policy $Z\in\Lambda$ we have
\begin{equation}\label{e:Z_v}
\frac{Z_T}{T} \leq \frac{v(x)}{T} +\ell +\frac{1}{T}  \int_0^T \sigma(X_s^Z)v’(X_s^Z)dB_s.
\end{equation}
The local martingale
$$
Q_T=\int_0^T \sigma(X_s^Z)v’(X_s^Z)dB_s
$$
has quadratic variation
$$
\frac{1}{T} [Q,Q]_T=\frac{1}{T} \int_0^T (\sigma(X_s^Z))^2 (v’(X_s^Z))^2 ds.
$$
By our assumptions and Lemma \ref{l:comparison} we have almost surely that
$$
X_s^Z\leq X_s,~s\geq 0.
$$
Then, almost surely
$$
\frac{1}{T} [Q,Q]_T=\frac{1}{T} \int_0^T (\sigma(X_s^Z))^2 (v’(X_s^Z))^2 ds \leq \frac{1}{T} \int_0^T (\sigma(X_s))^2 (v’(X_s))^2 ds.
$$
By the ergodic results from \cite{BS15} and the assumptions of the proposition, one has that almost surely
$$
\limsup_{T\rightarrow\infty} \frac{1}{T} [Q,Q]_T \leq \lim_{T\rightarrow\infty} \frac{1}{T} \int_0^T (\sigma(X_s))^2 (v’(X_s))^2 ds= \int g(x) m'(x)/m((0,\infty)) dx < \infty.
$$
The Law of Large Numbers for local martingales (see Theorem 1.3.4 from \cite{MAO}) yields that almost surely
$$
\lim_{T\rightarrow\infty }\frac{Q_T}{T}= 0.
$$
If we combine this with \eqref{e:Z_v} we get
$$
\liminf_{T\rightarrow\infty}\frac{Z_T}{T}\leq \ell.
$$
Finally, inequality \eqref{domopt} follows from \eqref{dom} and \eqref{OptLim}.
\end{proof}
\begin{rmk}
We make the following three remarks on the assumptions (1)-(3) needed in Theorem \ref{t:main_as}.
\begin{itemize}
  \item[(a)] If $0,\infty$ are unattainable and not attracting, i.e. for any $x\in(0,\infty)$ we have $\PP_x\left\{X_t\rightarrow 0\right\}=\PP_x\left\{X_t\rightarrow \infty\right\}=0$, and $m((0,\infty))<\infty$ then $X$ has a unique invariant probability measure with density $\frac{m'(\cdot)}{m((0,\infty))}$ on $(0,\infty)$. In terms of boundary behavior the points $0, \infty$ can be entrance or natural, and the natural boundaries have to be non-attracting.
  \item[(b)]  We note that the function $u$ defined in \eqref{e:u} satisfies $(\mathcal{A}u)(x) \leq \ell^\ast = \mu(b^\ast)b^\ast$ and $u'(x)\geq 1, x\in \R_+$.
  \item[(c)] Checking condition (6) reduces to looking at the function
  \begin{equation}\label{e:g}
\begin{aligned}
g(x)=\sigma(x) u'(x)
&= \begin{cases}
\sigma(x)\ell^\ast S'(x)m((0,x)), & \mbox{if $x\in(0,b^*)$}\\
\sigma(x) & \mbox{if $x\geq b^*$},
\end{cases}
\end{aligned}
\end{equation}
verifying that it is non-decreasing, and then checking whether $\int_0^\infty g^2(x) m'(x)\,dx<\infty$.
\end{itemize}

\end{rmk}
\begin{rmk}
Our work is related to \cite{JaZe06} where the authors consider the more general case where there are two controls. Consider the controlled diffusion
$$
d X_{t} =  \mu(X_{t})\,dt + \sigma(X_{t})\,dB_{t} +d\xi^+_t-d\xi^-_t
$$
where $\xi$ is a right-continuous process with left limits that has finite variation and is adapted. Fix a starting point $X(0)=x\in\R$. The paper \cite{JaZe06} is concerned with the the minimization of
$$
\limsup_{T\rightarrow\infty} \frac{1}{T}\E_x\left[\int_0^Th(X_s)\,ds+\int_{[0,T]}k_+(X_s)d\xi_s^+ +\int_{[0,T]}k_-(X_s)d\xi_s^-\right]
$$
and the almost sure minimization of
$$
\limsup_{T\rightarrow\infty} \frac{1}{T}\left[\int_0^Th(X_s)\,ds+\int_{[0,T]}k_+(X_s)d\xi_s^+ +\int_{[0,T]}k_-(X_s)d\xi_s^-\right].
$$
Here $h:\R\rightarrow\R$ is a given function that models the running cost resulting from the system's operation, while $k_+, k_-$ are given functions penalizing the expenditure of control effort.
We note that one of their assumption is that $0<\sigma^2(x)\leq C(1+|x|)$, which is more restrictive than what we have since in most population models $\sigma(x)= \sigma x$.
\end{rmk}

Our main result on the sign of the relationship between volatility and the optimal harvesting strategy is summarized in the following.
\begin{restatable}{thm}{mainv}\label{t:vol}
Increased volatility increases the optimal harvesting threshold $b^\ast$ and
decreases the long run average cumulative yield $\ell^\ast=\mu(b^\ast)b^\ast$.
\end{restatable}
\begin{proof}
Denote by $\tilde{b}$ the optimal harvesting threshold and by $\tilde{\ell}$ the maximal expected average cumulative yield associated with the more volatile dynamics characterized by the diffusion coefficient $\tilde{\sigma}(x)\geq \sigma(x)$ for all $x\in \mathbb{R}_+$
and let
$$
\tilde{\mathcal{A}}=\frac{1}{2}\tilde{\sigma}^2(x)\frac{d^2}{dx^2}+\mu(x)x\frac{d}{dx}
$$
denote the differential operator associated with the more volatile process. Let $u(x)$ be the solution of the free boundary problem \eqref{freebdry}. Because $u''(x)\leq 0$ we get
$$
(\tilde{\mathcal{A}}u)(x) = \frac{1}{2}(\tilde{\sigma}^2(x)-\sigma^2(x))u''(x) +(\mathcal{A}u)(x)\leq \ell^\ast
$$
for all $x\in \mathbb{R}_+$. Since we also have $u'(x)\geq 1$ we notice by combining Theorem \ref{t:main} and Lemma \ref{aux2} that $\tilde{\ell} \leq \ell^\ast$. However, since $\tilde{\ell} = \mu(\tilde{b})\tilde{b}$, $\ell^\ast = \mu(b^\ast)b^\ast$, and the optimal harvesting threshold is on the set where the drift is decreasing, we find $\tilde{b}\geq b^\ast$ which completes the proof of our claim.
\end{proof}

\section{Discounting and Harvesting: Connecting the Harvesting Problems}\label{s:discount}

The previous section focused on the optimal ergodic harvesting policy maximizing the expected (or almost sure) long-run average cumulative yield. It is naturally of interest to analyze in which way the optimal policy differs form the optimal policies suggested by models maximizing the expected \textit{present value} of the cumulative yield. To this end, let (cf. \cite{Al01})
\begin{align}\label{disc}
V_r(x) = \sup_{Z\in \Lambda}\mathbb{E}_x\int_0^\infty e^{-rs}dZ_s
\end{align}
denote the value of the harvesting policy maximizing the expected present value of the cumulative yield. Our objective is to characterize how the different problems are connected by relying on
an Abelian limit result first developed within singular stochastic control in the seminal paper \cite{Kar83} (see also \cite{Wee07} for a generalization).

In order to present our main findings on the connection between the two different approaches we first have to make a set of assumptions guaranteeing that the harvesting policy maximizing the expected present value of the cumulative yield is nontrivial. Define the function $\theta_r:\mathbb{R}\mapsto \mathbb{R}$ by
$$
\theta_r(x)=(\mu(x)-r)x,
$$
where $r>0$ denotes the prevailing discount rate. In addition to our assumptions on the boundary behavior of the population dynamics stated in Section \ref{s:model} we now assume the following.
\begin{asp} The function $\theta_r(x)$ satisfies
\begin{itemize}
  \item[(B1)] $\lim_{x\downarrow 0}\theta_r(x)\geq 0$ and $\lim_{x\rightarrow \infty}\theta_r(x)< -\varepsilon$, where $\varepsilon>0.$
  \item[(B2)] the function $\theta_r(x)$ attains a unique maximum at $\hat{x}_r\in (0,x_0^r)$, where $x_0^r = \inf\{x>0:\theta_r(x)=0\}$.
\end{itemize}
\end{asp}
\begin{rmk}
Note that if $\mu$ is continuous on $[0,\infty)$ then Assumptions \ref{A:ES} and\ref{A:suff} imply (B1) above. 
\end{rmk}

As was established in \cite{Al01}, one gets
\begin{align}\label{discdval}
V_r(x) = \begin{cases}
x + \frac{1}{r}\theta_r(x_r^\ast), &x\geq x_r^\ast,\\
\frac{\psi_r(x)}{\psi_r'(x_r^\ast)}, &x<x_r^\ast.
\end{cases}
\end{align}
The quantity $\psi_r(x)$ denotes the increasing fundamental solution of the differential equation $(\mathcal{A} u)(x)=ru(x)$. The optimal harvesting boundary $x_r^\ast = \argmin\{\psi_r'(x)\}\in (\hat{x}_r, x_0)$ is the unique root of the ordinary first order optimality condition $\psi_r''(x_r^\ast)=0$ which can be re-expressed as
\begin{align}\label{discopt}
\int_0^{x_r^\ast}\psi_r(z)(\theta_r(t)-\theta_r(x_r^\ast))m'(z)dz=0.
\end{align}
The value of the optimal harvesting policy $V_r(x)$ is monotonically increasing, concave, and twice continuously differentiable.
Moreover, increased volatility decreases the value of the optimal policy and expands the continuation region where harvesting is suboptimal by
increasing the optimal harvesting boundary $x_r^\ast$.

Under the optimal harvesting policy $Z^\ast$ the population density converges in law to its unique stationary distribution. In other words, $X_t^{Z^\ast}\Rightarrow \bar{X}_r$ as $t\rightarrow\infty$. The random variable
$\bar{X}_r$ is distributed on $(0,x_r^\ast)$ according to the density
$$
\mathbb{P}\left [\bar{X}_r\in dy\right] = \frac{m'(y)dy}{m((0,x_r^\ast))}.
$$

We can now establish the following limiting result
\begin{lm}\label{comparative}
Under our assumptions, increased discounting decreases the maximal expected present value of the cumulative yield and accelerates harvesting by decreasing the optimal harvesting boundary. Moreover, $\lim_{r\rightarrow 0+}x_r^\ast=b^\ast$ where $b^\ast$ is the optimal harvesting boundary from Theorems \ref{t:main} and \ref{t:main_as}.
\end{lm}
\begin{proof}
The monotonicity of the admissible harvesting strategy guarantees that increased discounting decreases the value of the optimal policy.
To see that it also accelerates harvesting by decreasing the optimal harvesting boundary we first observe that
under our assumptions the conditions of Lemma 3.1 in \cite{Alvarez2004} are met and, therefore,
\begin{align}\label{rep1}
\frac{\psi_r(x)-x\psi_r'(x)}{S'(x)}=\int_0^x\psi_r(t)\theta_r(t)m'(t)dt
\end{align}
for all $x\in \mathbb{R}_+$. Reordering terms
shows that \eqref{rep1} can be re-expressed as
\begin{align}\label{rep2}
u_r(x):=\frac{\psi_r'(x)}{\psi_r(x)}-x=\frac{S'(x)}{\psi_r'(x)}\int_0^x\psi_r(t)\theta_r(t)m'(t)dt.
\end{align}
On the other hand, if $\tau_y=\inf\{t\geq 0:X_t=y\}$ denotes the first hitting time to $y$, then the identity (cf. p. 18 in \cite{BS15})
$$
\mathbb{E}_x\left[e^{-r\tau_y};\tau_y<\infty\right]=\frac{\psi_r(x)}{\psi_r(y)}
$$
guarantees that if $\hat{r}>r$ then
$$
\frac{\psi_r(x)}{\psi_r(y)}\geq\frac{\psi_{\hat{r}}(x)}{\psi_{\hat{r}}(y)}
$$
for all $0<x <y <\infty$. Since
$$
\frac{\psi_r(x)}{\psi_r(y)} = \exp\left(-\int_{x}^{y}d\ln\psi_r(t)\right)
$$
we notice that
$$
\frac{\psi_r'(x)}{\psi_r(x)}\leq\frac{\psi_{\hat{r}}'(x)}{\psi_{\hat{r}}(x)}
$$
for all $0<x <y <\infty$. Consequently, $u_r(x)\leq u_{\hat{r}}(x)$ for all $x\in \mathbb{R}_+$.
Since $u_r(x_r^\ast)=\theta_r(x_r^\ast)$ and $\theta_r(x)\geq \theta_{\hat{r}}(x)$ for all $x\in \mathbb{R}_+$
we notice that
$$
u_{\hat{r}}(x_r^\ast)\geq \theta_r(x_r^\ast)\geq \theta_{\hat{r}}(x_r^\ast)
$$
implying that
$$
\int_0^{x_r^\ast}\psi_{\hat{r}}(t)\theta_{\hat{r}}(t)m'(t)dt\geq \theta_{\hat{r}}(x_r^\ast)\frac{\psi_{\hat{r}}'(x_r^\ast)}{S'(x_r^\ast)}
$$
and, therefore, that $x_{\hat{r}}^\ast \leq x_r^\ast$. This shows that higher discounting accelerates harvesting by decreasing the optimal threshold.

It remains to consider the limiting case where $r\downarrow 0$. To this end, consider the function
$$
F_r(x) = \int_0^x \frac{\psi_r(z)}{\psi_r(x)}(\theta_r(z)-\theta_r(x))m'(z)dz.
$$
Since
$$
\mathbb{E}_z[e^{-r\tau_x};\tau_x<\infty] = \frac{\psi_r(z)}{\psi_r(x)}
$$
for $z\leq x$ we notice by letting $r\downarrow 0$ and invoking our assumptions that
$$
 \lim_{r\downarrow 0}\frac{\psi_r(z)}{\psi_r(x)} = \mathbb{P}_z\left[\tau_x<\infty\right] = 1.
$$
Since $\lim_{r\downarrow 0}\theta_r(x)=\mu(x)$ we finally notice that
$$
\lim_{r\downarrow 0}F_r(x) = \int_0^x (\mu(z)z-\mu(x)x)m'(z)dz=f(x).
$$
The alleged claim now follows from \eqref{discopt} and Lemma \ref{aux2}.
\end{proof}

According to Lemma \ref{comparative} higher discounting accelerates harvesting and results into a lower expected asymptotic population density. Interestingly, the optimal harvesting threshold approaches the
one from the average cumulative yield setting as $r\rightarrow 0+$. It is clear that the same conclusion is not directly valid for the value of the optimal policy $V_r(x)$. However, there exists an Abelian limit connecting the value of the two seemingly different control problems. This connection is established in the following.
\begin{thm}
Under our assumptions,
$$
\lim_{r\rightarrow 0+}rV_r(x) = \mu(b^\ast)b^\ast = \ell^\ast
$$
for all $x\in \mathbb{R}_+$.
\end{thm}
\begin{proof}
Utilizing the fact that $rV_r(x_r^\ast)=\mu(x_r^\ast)x_r^\ast$ and reordering terms in \eqref{discdval} yields
\begin{align*}
rV_r(x) = \begin{cases}
\mu(x_r^\ast)x_r^\ast+r(x-x_r^\ast), &x\geq x_r^\ast,\\
\mu(x_r^\ast)x_r^\ast-r\int_{x}^{x_r^\ast}\frac{\psi_r'(t)}{\psi_r'(x_r^\ast)}dt, &x<x_r^\ast.
\end{cases}
\end{align*}
Since $\psi_r'(t)/\psi_r'(x_r^\ast)\in [1,\psi_r'(x)/\psi_r'(x_r^\ast)]$ for all $t\in[x,x_r^\ast]$ we find by invoking the limiting result of Lemma \ref{comparative} and the continuity of $\mu$ that
\begin{align*}
\lim_{r\rightarrow 0+}rV_r(x) = \begin{cases}
\mu(b^\ast)b^\ast, &x\geq x_r^\ast,\\
\mu(b^\ast)b^\ast, &x<x_r^\ast.
\end{cases}
\end{align*}
This completes the proof.
\end{proof}

\section{Applications}\label{s:applications}

\subsection{Verhulst-Pearl diffusion}

Assume that the unharvested population follows the standard Verhulst-Pearl diffusion
\begin{equation}\label{VerPea}
d X_t =  \mu X_t(1-\gamma X_{t})dt +  \sigma X_t dW_t,\quad X_0=x\in \mathbb{R}_+,
\end{equation}
where $\mu>0$ is the per-capita growth rate at low densities, $1/\gamma>0$ is the carrying capacity, and $\sigma>0$ is the infinitesimal variance of fluctuations in the per-capita growth rate. In this case
$$
S'(x)=x^{-\frac{2\mu}{\sigma^2}} e^{\frac{2\mu\gamma}{\sigma^2}x}
$$
and
$$
m'(x) = \frac{2}{\sigma^2}x^{\frac{2\mu}{\sigma^2}-2} e^{-\frac{2\mu\gamma}{\sigma^2}x}.
$$
Moreover, if $\mu>\sigma^2/2$ is satisfied, then
$$
m((0,x))=\frac{2}{\sigma^2}\left(\frac{\sigma^2}{2\mu \gamma}\right)^{\frac{2\mu}{\sigma^2}-1}\left(\Gamma\left(\frac{2\mu}{\sigma^2}-1\right)-\Gamma\left(\frac{2\mu}{\sigma^2}-1,\frac{2\mu\gamma x}{\sigma^2}\right)\right).
$$
We note that in this case Assumptions \ref{A:ES}, \ref{A:suff} and the conditions of Theorem \ref{t:main_as} hold (see \cite{EHS15} for a thorough investigation of \eqref{VerPea}).
Consequently,
$$
\ell^\ast =\frac{1}{2}\frac{\sigma^2}{\gamma} \left(\frac{\sigma^2}{2\mu}\right)^{1-\frac{2\mu}{\sigma^2}}\argmax\left\{\frac{(\gamma x)^{\frac{2\mu}{\sigma^2}} e^{-\frac{2\mu\gamma}{\sigma^2}x}}{\Gamma\left(\frac{2\mu}{\sigma^2}-1\right)-\Gamma\left(\frac{2\mu}{\sigma^2}-1,\frac{2\mu\gamma x}{\sigma^2}\right)}\right\}.
$$
The optimal boundary $b^\ast$ reads as $b^\ast=\rho^\ast \gamma^{-1}$, where $\rho^\ast$ is the unique root of the equation
$$
\left(\frac{2\mu\rho^\ast}{\sigma^2}\right)^{1-\frac{2\mu}{\sigma^2}} (1-\rho^\ast) e^{\frac{2\mu\rho^\ast}{\sigma^2}}\frac{2\mu}{\sigma^2}\left(\Gamma\left(\frac{2\mu}{\sigma^2}-1\right)-
\Gamma\left(\frac{2\mu}{\sigma^2}-1,\frac{2\mu\rho^\ast}{\sigma^2}\right)\right)=1.
$$
This shows that the optimal harvesting boundary is directly proportional to the carrying capacity.

As was shown in \cite{AS98}, in this case
the harvesting boundary maximizing the expected present value of the cumulative yield constitutes the root of the equation $\psi_r''(x_r^\ast)=0$, where $r>0$ denotes the prevailing discount rate,
$$
\psi_r(x) = (\gamma x)^{\alpha_1} \hat M\left(\alpha_1,1+\alpha_1-\alpha_2,\frac{2\mu\gamma x}{\sigma^2}\right),
$$
$\hat M$ denotes the Kummer confluent hypergeometric function,
$$
\alpha_{1} := \frac{1}{2} - \frac{\mu}{\sigma^{2}} + \sqrt{\left(\frac{1}{2} - \frac{\mu}{\sigma^{2}}\right)^{2}
+ \frac{2 r}{\sigma^{2}}} > 0,
$$
and
$$
\alpha_{2} := \frac{1}{2} - \frac{\mu}{\sigma^{2}} - \sqrt{\left(\frac{1}{2} - \frac{\mu}{\sigma^{2}}\right)^{2}
+ \frac{2 r}{\sigma^{2}}} < 0.
$$
We notice again that as in the ergodic setting, the optimal threshold is directly proportional to the carrying capacity.

The optimal harvesting threshold is illustrated for two different volatilities as a function of the discount rate in Figure \ref{fig1a} under the assumptions that $\mu = 0.1,\gamma = 0.001$.
\begin{figure}[h!]
\begin{center}
\includegraphics[]{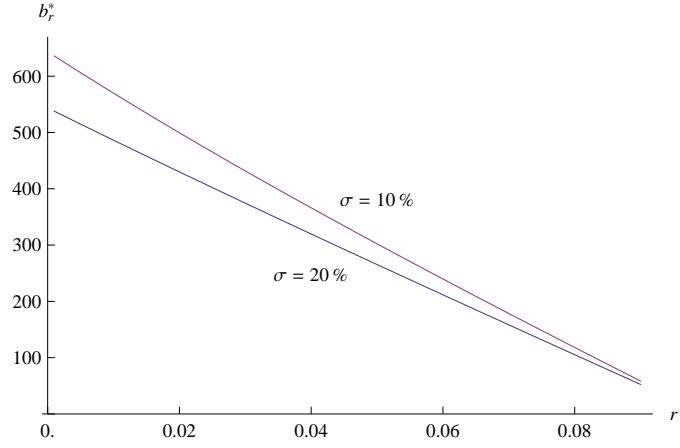}
\end{center}
\caption{\small{The optimal harvesting boundary as a function of the discount rate.}}\label{fig1a}
\end{figure}

\subsection{Logistic Diffusion}

An alternative logistic population growth model was studied in \cite{LO97} and in \cite{A00}.
The dynamics is characterized by the stochastic differential equation
\begin{equation}\label{LogDif}
dX_t=\mu X_t(1-\gamma X_t)dt + \sigma X_t (1-\gamma X_t)dW_t,\quad X_0=x\in (0,1/\gamma).
\end{equation}
As was established in \cite{LO97}, this SDE has a unique strong solution defined for all $t\geq 0$. In this case we know that
$$
S'(x)= \left(\frac{x}{1-\gamma x}\right)^{-\frac{2\mu}{\sigma^2}},
$$
and
$$
m'(x)=\frac{2}{\sigma^2}x^{\frac{2\mu}{\sigma^2}-2}(1-\gamma x)^{-\frac{2\mu}{\sigma^2}-2}.
$$
Moreover, if $\mu>\sigma^2/2$ is satisfied, then for any $x>0$
$$
m((0,x))=\frac{2}{\sigma^2}\gamma^{1-\frac{2\mu}{\sigma^2}}B\left(\gamma x, \frac{2\mu}{\sigma^2}-1,-\frac{2\mu}{\sigma^2}-1\right),
$$
where $B$ denotes the incomplete beta-function.
One can see that in this setting Assumptions \ref{A:ES}, \ref{A:suff} and the conditions of Theorem \ref{t:main_as} hold.
Consequently,
$$
\ell^\ast =\frac{1}{2}\sigma^2 \gamma\argmax\left\{\left(\frac{\gamma x}{1-\gamma x}\right)^{\frac{2\mu}{\sigma^2}}\frac{1}{B\left(\gamma x, \frac{2\mu}{\sigma^2}-1,-\frac{2\mu}{\sigma^2}-1\right)}\right\}
$$
demonstrating that the optimal harvesting boundary is directly proportional to the carrying capacity in this case as well.
Standard differentiation now shows that the harvesting threshold maximizing the expected average cumulative yield is $b^\ast = \rho^\ast \gamma^{-1}$ where $\rho^\ast$ constitutes the unique root of the equation
$$
\left(\frac{\rho^\ast}{1-\rho^\ast}\right)^{\frac{2\mu}{\sigma^2}}= \rho^\ast(1-\rho^\ast)\frac{2\mu}{\sigma^2}B\left(\rho^\ast, \frac{2\mu}{\sigma^2}-1,-\frac{2\mu}{\sigma^2}-1\right).
$$
As was shown in \cite{A00}, in this case
the harvesting boundary maximizing the expected present value of the cumulative yield is the unique solution of $\psi_r''(x_r^\ast)=0$, where
$$
\psi_r(x) = \left(\frac{\gamma x}{1 - \gamma x}\right)^{\alpha_{1}} F\left(a, b, c; -\frac{\gamma x}{1 - \gamma x}\right),
$$
$F$ is the standard hypergeometric function,
$$
a := 1 - \frac{\alpha_{2}}{2} + \frac{\alpha_{1}}{2} - \frac{1}{2}\sqrt{(\alpha_{2}^{2} -
2 \alpha_{2}(2 + \alpha_{1})) + (2-\alpha_{1})^{2}},
$$
$$
b := 1 - \frac{\alpha_{2}}{2} + \frac{\alpha_{1}}{2} + \frac{1}{2}\sqrt{(\alpha_{2}^{2} -
2 \alpha_{2}(2 + \alpha_{1})) + (2-\alpha_{1})^{2}},
$$
and
$$
c := 1 - \alpha_{2} + \alpha_{1}.
$$
We notice again that, as in the ergodic setting, the optimal threshold is directly proportional to the carrying capacity.

The optimal harvesting threshold is illustrated for two different volatilities as a function of the discount rate in Figure \ref{fig1} under the assumptions that $\mu = 0.1,\gamma = 0.001$.
\begin{figure}[h!]
\begin{center}
\includegraphics[]{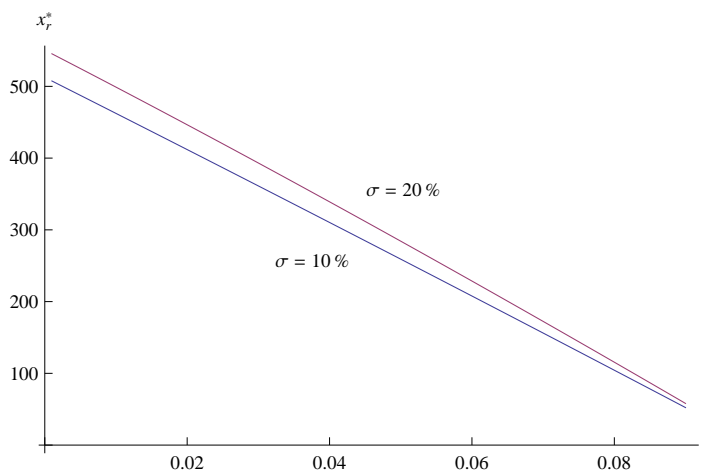}
\end{center}
\caption{\small{The optimal harvesting boundary as a function of the discount rate.}}\label{fig1}
\end{figure}

\section{Discussion}\label{s:discussion}

We investigated the optimal ergodic harvesting strategies of a population $X$ whose dynamics is given by a general one-dimensional stochastic differential equation. The theory we develop for optimal sustainable harvesting includes the risks of extinction from evironmental fluctuations (environmental stochasticity) as well as from harvesting. However, in contrast to most of the literature, we do not work with discount factors (see \cite{LES94, LES95, AS98, LO97, A00}) or maximal harvesting rates (see \cite{HNUW18}). Instead, we concentrate on policies aiming to the maximization of the average cumulative yield. We proved that the optimal policy is of the same local time push type as in the discounted setting. Since the optimal threshold at which harvesting becomes optimal is a decreasing function of the prevailing discount rate, our results unambiguously demonstrate that
policies based on ergodic (sustainable) criteria are more prudent and imply higher population densities than models subject to discounting. Our results show higher stochastic fluctuations negatively impact the population densities -- this provides rigorous mathematical support for the arguments developed in \cite{LES94} based on approximation arguments.

There are at least two directions in which our analysis could be extended. As was originally established in \cite{Kar83} in a model based on controlled Brownian motion, the value of the optimal ergodic policy is associated with the value of the finite horizon control problem
$$
J(T,x) = \sup_{Z\in \Lambda} \mathbb{E}_{x}\left[Z_T\right],
$$
where $T>0$ is a known fixed finite time horizon. It would be of interest whether an analogous limiting result connects $J(T,x)/T$ to $l^\ast$ in the general setting. 

The present analysis focuses on the harvesting of a single unstructured population. It would relevant to increasing the dimensionality of the considered model and introduce interactions into the dynamics governing the evolution of the population stock (see for example the population dynamics models from \cite{HN18, HN18a, SBA11}). In light of the studies \cite{LuOk01,AlLuOk16} this latter problem seems to be very challenging and is left for future considerations.

\bibliographystyle{amsalpha}
\bibliography{harvest}

\end{document}